\documentclass [12]{amsart}
\usepackage[utf8]{inputenc}
\usepackage[T1]{fontenc}
\usepackage{geometry} 
\usepackage{tikz-cd}
\usepackage{amsbsy}
\usepackage{amssymb,latexsym} 
\usepackage{amsmath,amsthm} 
\usepackage[russian,english,]{babel}
\usetikzlibrary{cd}
 \usepackage{mathtools}
 
\usepackage{hyperref}
  
 \newcommand{\cQ}{\mathcal{Q}}
\newcommand{\cA}{\mathcal{A}}
\newcommand{\cD}{\mathcal{D}}

\newcommand{\cI}{\mathcal{I}}

\newcommand{\cV}{\mathcal{V}}
\newcommand{\cX}{\mathcal{X}}

\newcommand{\fA}{\mathfrak{A}}

\newcommand{\fM}{\mathfrak{M}}

\newcommand{\A}{\mathbb{A}}

\newcommand{\N}{\mathbb{N}}
\newcommand{\R}{\mathbb{R}}
\newcommand{\Q}{\mathbb{Q}}
\newcommand{\Z}{\mathbb{Z}}

\newcommand{\Cc}{\mathbb{C}}

\newcommand{\Pp}{\mathbb{P}}

\newcommand{\bA}{{\mathbf A}}

\newcommand{\bK}{{\mathbf K}}

\def\be{\mathbf{e}}
\def\bp{\mathbf{p}}
\def\bq{\mathbf{q}}
\def\bu{\mathbf{u}}
\def\bv{\mathbf{v}}
\newcommand{\rA}{{\mathrm A}}
\newcommand{\rC}{{\mathrm C}}

\newcommand{\rK}{{\mathrm K}}

\newcommand{\rR}{{\mathrm R}}

\newtheorem{example}{Example}
\newtheorem{thm}{Theorem}
\newtheorem{lem}{Lemma}
\newtheorem{cor}{Corollary}
\newtheorem{prop}{Proposition}

\theoremstyle{definition}
\newtheorem{dfn}{Definition}

\theoremstyle{remark}
\newtheorem{rem}{Remark}
\newtheorem{exa}{Example}
\theoremstyle{remark}

\title[Exponential varieties, statistical manifolds and Frobenius structures]{A note on exponential varieties, statistical manifolds and Frobenius structures}
\author{No\'emie C. Combe } 
\address{Max Planck Institute for Mathematics in the Sciences,
Inselstr. 22,04103, Leipzig}
	
\email{noemie.combe@mis.mpg.de}
\thanks{I acknowledge the Minerva Grant from the Max Planck Society for supporting my work.}


\date{\today}

\begin{document}

\maketitle

\begin{abstract}
New relations between algebraic geometry, information theory and Topological Field Theory are developed.
One considers models of databases subject to noise i.e. probability distributions on finite sets, related to exponential families. We prove explicitly that these manifolds have the structure of a pre-Frobenius manifold, being a pre-structure appearing in the process of axiomatisation of Topological Field Theory. On one hand, this allows us to develop relations to algebraic geometry, by proving explicitly that a statistical pre-Frobenius manifold forms an algebraic variety over $\Q$ (i.e. $\Q$-toric variety). On the other hand, this allows further developments of recent results concerning the hidden symmetries of those objects. Using classical web theory, it has been shown that those symmetries have the structure of Commutative Moufang Loops. Our result allows to develop more algebraically this statement, in a two-fold way. First, from an algebraic point of view it follows that statistical pre-Frobenius manifolds are equipped with algebraizable webs. Secondly, from the differential geometry point of view, it follows that these webs are hexagonal and isoclinic. This statement is important since it directly impacts the geometric properties of the {\it statistical data}, which are tightly related to the webs. Hence, this allows deeper connections to the branch of algebraic statistics, which is concerned with the development of techniques in algebraic geometry, commutative algebra, to address problems in statistics and its applications. Examples are provided and discussed. 
\end{abstract} 
\tableofcontents
\section{Introduction}
The manifold of probability distributions, related to exponential families is considered. It forms a pre-Frobenius manifold and whenever its structure connection $\nabla_{\lambda}$ is flat (for $\lambda$ an even parameter) it has a Frobenius manifold structure. We call this manifold for simplicity {\it statistical (pre-)Frobenius manifold}. Frobenius manifolds play a key role in the axiomatisation of Topological Field theory~\cite{Du,Lu}

In addition to giving a geometric proof of the existence of a pre-Frobenius structure on this type of manifold, we show a deep relation between statistical pre-Frobenius manifolds~\cite{CoMa20} and algebraic varieties over $\Q$. In particular, when the sample space is discrete, then the statistical pre-Frobenius manifold is a rational toric variety~\cite{PRW01,MaPi10,PiRo11,MSUZ14}, which forms a type of algebraic variety~\cite{Sha72, Ha77, Da78}. 

Toric varieties have been an effervescent research topic, and among the most important of algebraic geometry for many years \cite{Br04,CLS11,Da78, Fu, She22, Stu}. Toric varieties continue to be at the heart of a wide range of interests and conjectures, including in the mathematical vision of the mirror problem. Furthermore, one important conjecture was given by Manin~\cite{LoMa00}. This  concerns a conjectural distribution of rational points on an algebraic variety (Fano variety) relative to a suitable height function. This has been proved as being true for some cases such as toric varieties~\cite{BaTsch}. Recently, a result by the author~\cite{Co} shows that this conjecture can be extended to certain objects of information transmission and that the Manin conjecture is also true in that case.  

Our approach in this paper develops furthermore the dialogue between algebra and structures related to transmitting information and data collection. This benefits to both disciplines. On the one hand, algebra provides a powerful tool set for addressing statistical problems. On the other hand, it is rarely the case that algebraic techniques are ready-made to address statistical challenges, and usually new algebraic results need to be developed~\cite{Bernd}. This is what precisely motivates our investigations in this paper. 

Our results in this paper present a ``Part 3'' of recent developments in~\cite{CoMaMaA, CoMaMaB}. In these papers it has been shown that symmetries of various structures related to transmitting information (and data collection) are embodied in various classes of algebraic structures.  Such algebraic structures include loops and Moufang loops. In~\cite{CoMaMaA} an explicit proof is given of the fact that symmetries of probability distributions over a finite set have a Moufang loop structure. Our aim is to show the existence of an algebraic aspect of these works, giving possible applications in the realm of algebraic statistics. 

The underlying tool to see these algebraic structures appearing is to use web theory. This is turned out historically and recently to be a successful tool for unraveling analytical/ algebraic (local) invariants related to a given manifold (or the algebraic variety under consideration). For instance, a very recent result on webs in relation to another class of Frobenius manifolds has allowed to prove the existence of hidden symmetries of the Grothendieck--Teichm\"uller group~\cite{CoKa} and of the motivic Galois group. 

Going back to hidden symmetries of probability distributions over a finite set, this result has lead to the proof that the F-manifold structure on this space can be described in terms of differential three-webs and Malcev algebras. We develop this aspect furthermore, by proving the additional property: statistical pre-Frobenius manifolds are equipped with algebraizable webs. 

This statement is important because it impacts the geometric properties of the {\it statistical data}. Furthermore, it allows deeper connections to the branch of algebraic statistics, which is concerned with the development of techniques in algebraic geometry, commutative algebra, and combinatorics, to address problems in statistics and its applications. 

Furthermore, since webs have also differential geometry origins this allows a supplementary connection to differential geometry. From the more projective differential geometry side, one has a local invariant property that those webs are hexagonal and isoclinic. 

Finally, we can refer to one more aspect of algebraic webs as follows. Algebraic webs are closely related to algebraic loops and to quasigroups (these satisfy the group axioms concerning the identity and the inverse but the associativity condition is not required). 
The left loops (loops with a left type of multiplication) are in bijection with left homogeneous spaces (which are equipped also with a left type of multiplication). Therefore, one can say that left loops and  left homogeneous spaces describe the same structure on a given affine connection manifold (see \cite{No} for a definition of affine connection manifolds). 

The space probability distributions on a finite set have precisely the affine connection manifolds~\cite{CoMa20} structure.
A property of these types of manifolds is that one can introduce a specific type of local loop in the neighborhood of any point (see \cite{Kika} for details on the properties of those loops). These loops are uniquely determined by means of the parallel transport along a geodesic~ \cite{Sa1,Sa2}. Now, this family of local loops defines a covering of the manifold and uniquely determine all affine connections. So, applying algebraic webs here provides us with important information concerning the geometrical structure of the manifold under investigation.

\subsection{Webs and their geometry}
Three-webs were initially introduced by E. Cartan~\cite{Ca08,Ca28,Ca37,Ca60,Ca62,CaH67} as a movable frame method, and played a crucial role in differential-geometric studies of the 20th century. This theory also was investigated by W. Blaschke ~\cite{Bl28, Bl33,Bl38, Bl55}.

The French and German names for the concept are respectively ``tissu'' and ``Gewebe'', to recall the fabrication et optimization for cut of fabric~\cite{Tch78}. These words with these meanings convey a good intuitive picture of the kind of geometric structure involved. An easy example of a 3-web of curves in the plane is given by the three systems of lines $x = const$, $y = const$, and $x+y = const$. More generally, the last can be replaced by any smooth function $f$, verifying $f(x,y)=const$.

\begin{figure}[h] \label{F:hex}
   \begin{center}
    \includegraphics[width=0.4\textwidth]{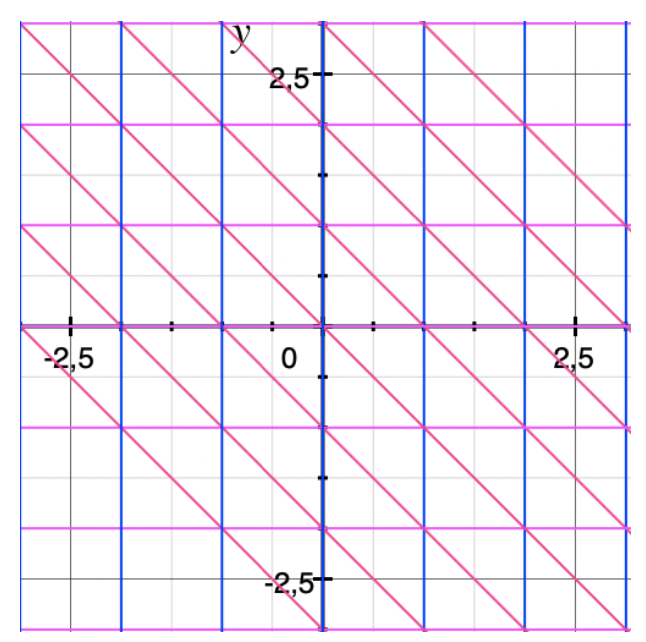}
 \caption{Three web given by $x = const,\, y = const$, and $x+y = const$ on a 2-dimensional domain $\cD$}
\end{center}
\end{figure}

Many of the papers of W. Blaschke concerned Klein's famous ``Erlangen program'' (1872). According to this program, the philosophy was that every geometry studies properties of geometric figures that are invariant under transformations composing a certain group. This philosophy was extended to web geometry. Blaschke and Thomsen (see~\cite{Bl28} and~\cite{Th27}) started to study 3-webs after they realized that the configuration of three foliations of curves in the plane has local invariants.

Later on, S. S Chern~\cite{Cher36} in (1936) introduced a reduction of the linear coframe bundle over the so-called web manifold, corresponding to the 3-web structure. This version of the linear coframe bundle was equipped with what we now call the Chern connection. This connection is invariant with respect to local differentiable equivalences of 3-web.  

Further developments of this approach were carried out by M. A. Akivis (1969) and his school. Most of the results are summarised in~\cite{AkSh92}. Some results were generalised in~\cite{Go88} to $(n+1)$-webs (where one has $n+1$ foliations instead of three). 

In particular, in the Akivis school they used this approach of adapted coframe bundle to prove the existence of {\it invariants of 3-webs}, and clarified the relation between those 3-webs and its corresponding algebraic isotopy invariant identities of the coordinate loop. Namely, in~\cite{AkSh92} the relation between web geometry with the theory of (smooth local differentiable) quasigroups and loops as well as with the theory of binary-ternary algebras was explicitly shown.

Note also that web geometry was instrumental in the solution of the algebraization problem for (a number of) $d$ submanifolds in a real projective space. This problem was solved for instance in~\cite{Ak83,Go82,Woo82,Woo84}. 

To give an intuition of the idea behind three webs, let us consider a 2-dimensional manifold $\cD$. Take three regular families of smooth curves in $\cD$, that are in general position. We say that they form a 3-web in $\cD$. We say that two 3-webs are equivalent one to another if there exists a local diffeomorphism which maps the families of one web into the families of another one. In a more sophisticated manner, we can say that three webs are three foliations of a manifold, where leaves are in general position.

\smallskip 

Section~2: presents the definition of the manifold of probability distributions, related to exponential families. We prove that it is a (pre-)Frobenius manifold. Moreover, if it satisfies a flatness condition on the structure connection then it is Frobenius. Such a statistical manifold satisfying the latter condition goes by the name of {\it statistical Frobenius manifold} or fourth Frobenius manifold (referring to the fact that {\it only three} other sources of Frobenius manifolds are known). 

\smallskip 

Section~3: After a short introduction to the notion of toric varieties, we proceed to the proof of the main statement: the pre-Frobenius statistical manifold is a toric variety, if the sample space is discrete. 

\smallskip 

Section~4: notions coming from web theory are recalled. In particular, the previous theorem on pre-Frobenius statistical manifolds implies that webs, given by the Chern-Griffiths construction, are algebraizable and satisfy the hexagonality property. Consequences of our theorem concerning statistical data and in particular their geometric properties are given.

\medskip

\section{(pre-)Frobenius statistical manifolds}~\label{S:2}
We prove explicitly that there exists a pre-Frobenius manifolds structure \cite{Man99,Du,Lu} appearing in exponential statistical manifolds. For the sake of clarity, we recall the classical definition on these objects (this is exposed in section 2.1.). However, a different definition (more geometrical) is necessary to see the structure of (pre-)Frobenius manifolds emerging. This resides in the approach given by H. Shima \cite{Shi77} and is developed in section 2.2.

\subsection{Exponential statistical models/ manifolds}~\label{S:2.1}
A statistical model (or manifold) can be considered in two equivalent ways. A first way is to consider it as a parametrized family of probabilities, denoted $P_\theta$ and being absolutely continuous with respect to a $\sigma$-finite measure $\mu$. The second way is to consider it as the parametrized family of probability distributions $S=\{p(x;\theta)\}$, where $p(x;\theta) =\frac{dP_\theta}{d\mu}$ is the Radon--Nikodym derivative of $P_\theta$ w.r.t. $\mu$. This space comes equipped with:
\begin{itemize}
\item the so-called {\it canonical parameter} given by $\theta= (\theta^1,\dots, \theta^n)\in \R^n$;
\item a family of random variables $x=\{x_i\}$  on a sample space $\Omega$;  
\item $p(x;\theta)$, the density of probability of the random variable $x$ parametrized by $\theta$ w.r.t some measure $\mu$ on $\Omega$.   
\end{itemize}
  
A family $S=\{p(x;\theta)\}$ of distributions is an {\it exponential family} if the density functions can be written in the following way:   
  \[p(x;\theta)= \exp(\theta^ix_i-\Psi(\theta)),\] where   
  \begin{itemize}
  \item the symbol $\Psi(\theta)$ stands for a potential function. It is given by $\Psi(\theta)=\log\int_{\Omega}\exp\{\theta^ix_i\}d\mu$. 
  \item the parameter $\theta$ and  $x=(x_i)_{i\in I}$ ($I$ is a finite set) have been chosen adequately;
\item the  canonical parameter satisfies $\Theta:=\{\theta\in \R^d: \Psi(\theta)<\infty\}$.
\end{itemize}
It is natural to include in the statistical model a structure of $n$-dimensional manifold, whenever $p(x;\theta)$ is smooth enough in $\theta$.

\smallskip 

We adopt now a more geometrical definition of the exponential statistical manifolds. This allows us to prove explicitly the existence of a (pre-)Frobenius manifold structure on this type of manifold. 

\smallskip 

A statistical structure on a differentiable manifold $M$ is a pair $(g, \rC)$, where $g$ is a metric tensor and $\rC$ a totally symmetric cubic (0,3)-tensor (i.e., 3-covariant tensor) called the Amari-Chentsov tensor~\cite{Am85,Chen64,Chen82,La87}. Such a manifold $M$ can be equipped with the Levi--Civita connection $\nabla^0$ for $g$, which is the unique torsion free affine connection compatible with the metric $g$, satisfying the following equation, for any triple of vector fields $(X, Y, Z)$ :
\[ X(g(Y, Z)) = g(\nabla^0_{_X}Y, Z) + g(Y, \nabla^0_{_X} Z), \]
or equivalently the Koszul formula:
\[\begin{aligned}
2g(\nabla^0_{_X}Y,Z)&=X(g(Y,Z))+Y(g(X,Z))-Z(g(X,Y))\\
&+g([X,Y],Z)-g([X,Z],Y)-g([Y,Z],X)\end{aligned}.\]

\smallskip 

 Given the Levi--Civita connection $\nabla^0$ for $g$ on $M$,
we have a pencil $\{\nabla^\alpha\}$ of $\alpha$-connections depending on a parameter $\alpha$ defined as: 
\begin{equation}\label{E:1} 
g(\nabla^{\alpha}_{_X}(Y),Z) :=g(\nabla^{0}_{_X}(Y),Z) -\frac{\alpha}{2}\rC(X,Y,Z),
\end{equation}
 where $X,Y,Z$ are vector fields in the tangent sheaf $T_M$. 
 
 These  $\alpha$-connections being affine and torsion free  allow to derive the notion of covariant derivative, parallel transport  and geodesics. 
  Moreover any pair  $(\nabla^\alpha,\nabla^{-\alpha})$ defines a pair of conjugate connections  $(\nabla,\nabla^\star)$ with respect to the metric $g$, that is
  \begin{equation}\label{E:2} 
X(g(Y, Z)) = g(\nabla_{_X}Y, Z) + g(Y, \nabla^\star_{_X} Z),
\end{equation}
 where $ \nabla=\nabla^\alpha$ and $\nabla^\star=\nabla^{-\alpha}$.
 
 By \eqref{E:1} the Amari--Chentsov tensor can be expressed in terms of a (1,2)-tensor difference tensor $\rK$, as:
 \begin{equation}\label{E:3} 
  \rC(X,Y,Z)= g(K(X,Y),Z), \quad \rK(X,Y)=\frac{2}{\alpha}(\nabla^{-\alpha}_{_X}-\nabla^{\alpha}_{_X})Y
\end{equation}
 
Now, the curvature tensor $\rR^\alpha$ of an  $\alpha$-connection $\nabla^\alpha$ is a (1, 3)-tensor given by:
\begin{equation}\label{E:4}
g(\rR^\alpha(X, Y )Z ,W)= g([\nabla^\alpha_{_X},\nabla^\alpha_{_Y} ]Z,W) - g(\nabla^\alpha_{_{[X,Y ]}}Z,W),
\end{equation}
  where $X,Y,Z,W$ are vector fields on $M$.
  
In the pencil we have $\rR^\alpha=\rR^{-\alpha}$, hence if an $\alpha$-connection is flat (or locally flat) then, the $-\alpha$-connection is flat as well. For example in the important case of  exponential families the $\alpha$-connections are flat if $\alpha= \pm 1$.

This construction provides the possibility of a generalization. In particular, from this construction emerges directly the (pre-)Frobenius manifold structure .

\subsection{Exponential statistical Pre-Frobenius and Frobenius manifolds}\label{S:2.2} 
Let $M$ be a manifold in one of the standard categories ($C^\infty$, analytic, algebraic, formal, $\cdots$). 
We recall the classical definition of a pre-Frobenius manifold and focus on the $\mathbb{R}$ case. 

A pre-Frobenius structure~\cite{Man99} on a manifold $M$ follows from the existence of the following data:
  
 -- an affine flat connection (in this context this is to be interpreted as having a complete atlas whose transition functions are affine linear),

-- a compatible metric $g$,

-- an even symmetric rank three tensor $\rA: T_M\times T_M\times T_M\to \mathbb{R}$,

-- a symmetric bilinear multiplication operation 
 $\circ_{\rA,g}: T_M \times T_M \to T_M$. Furthermore, the metric is invariant with respect to multiplication:
\begin{equation}\label{E:5} 
 \rA(X,Y,Z)=g(X,Y\circ Z)=g(X\circ Y, Z), \quad  X,Y,Z \in T_M. 
\end{equation}

A Frobenius structure is a pre-Frobenius manifold such that associativity and potentiality axioms are satisfied. 

-- The associativity condition is satisfied if the multiplication is associative i.e. 
\begin{equation}\label{E:6} 
(X\circ Y)\circ Z=X \circ (Y \circ Z),
\end{equation}
 
 -- The potentiality condition is satisfied if everywhere locally $\rA$ admits a potential, i.e. a local even function $\phi$ such that for any flat local tangent fields $\partial_a,\partial_b,\partial_c$ we have
 \[
 \rA(\partial_a,\partial_b,\partial_c)=\partial_a\partial_b\partial_c\phi.
\]
Given the Levi--Civita connection $\nabla_0$ for $g$ on $M$, we have $\nabla_\lambda$, a pencil of connections depending on an even parameter $\lambda$. This pencil of connections is defined as: 
\begin{equation}\label{E:7} 
\nabla_{\lambda,X}(Y) :=\nabla_{0,X}(Y) +\lambda  X\circ Y,
\end{equation}

Going back to the statistical structure, the symmetric difference tensor $K$ allows the definition of the symmetric multiplication operation $\circ_{\cA,g}$ on the tangent sheaf $T_M$. Indeed, given that $\partial_a \circ \partial_b =\sum_c \rA_{ab}^c\partial_c$ we obtain from \eqref{E:3}: 
\[\circ_{\rA,g}: T_M \times T_M \to T_M\]
\[(\partial_a,\partial_b)\mapsto \sum_c \rK_{ab}^c\partial_c.\] 
 The tensor $\rA$ is a symmetric rank three tensor, which can be identified to $-2\rC$ in the context of information geometry. 
We can thus say that:

\begin{lem}\label{L:1}
 A statistical structure  $(M,g,\{\nabla_\lambda\}, \rC)$, equipped with the pencil $\{\nabla_\lambda\}$ satisfies the pre-Frobenius manifold definition.  
\end{lem}
\begin{proof}
The proof follows from the discussion above.
\end{proof}

Now, the following statement puts a final point to the discussion on Frobenius manifolds and statistical structure.
\begin{thm}\label{T:1}
Let $(M,g, \{\nabla_\lambda\}, \rA)$ be a pre-Frobenius manifold coming from the statistical structure, as defined above. 
Suppose that $\{\nabla_\mu\}$ is a sub-pencil of $\{\nabla_\lambda \}$ formed only from those connections being flat. 
Then, the manifold is equipped with a Frobenius structure. 

Reciprocally,  if $(M,g, \{\nabla_\mu\}, \rA)$  is a Frobenius manifold, then:

\noindent 1)  the family of connections $\{\nabla_\mu\}$ on M are necessarily flat,
\,
or 

\noindent 2)  the family $\{\nabla_\mu\}$ of connections belong to a sub-pencil of connections being flat.
\end{thm}

In particular if $(M, g, \nabla, \rA)$ is flat with respect to the connection $\nabla$, then $(M, g, \overline{\nabla}, \rA)$ is also flat with respect to $\overline{\nabla}$. This type of hidden symmetry can be recovered also by geometric means using biaxial geometry (see O. Mayer~\cite{May38} for the original definition on biaxial geometry. Initially, biaxial geometry was used for skew curves in the projective space.  Biaxial collineation included an involution. See~\cite{CoMa20} for the case of biaxial geometry in the context of paracomplex numbers and $F$-manifolds). 

One can furthermore comment on the flatness condition of $\nabla$ for the statistical structure.
\begin{lem}
If $(M, g, \nabla , \rA)$ is flat with respect to $\nabla$ then the statistical structure is Hessian. 
\end{lem}
\begin{proof}If $(M, g, \nabla,\rA)$ is flat with respect to $\nabla$ then, by theorem \ref{T:1} it is Frobenius. In particular, one of the Frobenius axioms imply potentiality, which means that $\rA(\partial_a,\partial_b,\partial_c)=\partial_a\partial_b\partial_c\phi$, where $\phi$ is a potential function, defined (everywhere) locally. In particular from the relation 
\[\rA(X,Y,Z)=g(X \circ Y,Z)=g(X, Y\circ Z)\] it follows that $g(\partial_a,\partial_b)=\partial_a\partial_b\phi$ hence $g$ is the Hessian of the potential function $\phi$.
\end{proof}

\subsection{Recollections on algebraic and toric varieties}~\label{S:2.3}
As a prelude to this section we start by recalling basic notions and properties of toric varieties. Once this is done we can proceed to the proof of our statement, bridging pre-Frobenius statistical manifolds (of exponential type) to algebraic varieties and, in particular, to toric varieties.

Let $\bK$ be a base field. We consider toric varieties over $\bK$. Many works have been developed in the context of complex numbers (see \cite{Fu} for a classical reference). However, toric varieties have been defined more generally over any number field (see for instance \cite{Da78} for a general exposition and for example \cite{BaTsch} for toric varieties over number fields).

Roughly speaking, a smooth $n$-dimensional toric variety is an algebraic variety $\cX$, together with a collection of charts $x^{(\alpha)}: U_\alpha\to \bK^n$, 
such that on the intersections of $U_\alpha$ with $U_{\beta}$ the coordinates $x^{(\alpha)}$ must be Laurent monomials in the $x^{(\beta)}$.

Supposing $\cX$ is a toric variety, we fix one chart $U_0$ with coordinates $x_1,\cdots, x_n$. Then the coordinate functions $x^{(\alpha)}$ on the remaining $U_\alpha$ (and monomials in them) can be represented as Laurent monomials in $x_1, . . ., x_n$ . 
Furthermore, if $f: U_\alpha \to \bK$ is a "regular" function on $U_\alpha$, i.e. a polynomial in the variables $x^{(\alpha)}_1,\cdots,x^{(\alpha)}_n$, 
then $f$ can be represented as a Laurent polynomial in $x_1,\cdots, x_n$ (as a finite linear combination of Laurent monomials). The monomial character of the coordinate transformation is reflected in the fact that the regularity condition for a function $f$ on the chart $U_\alpha$ can be expressed in terms of the support of the corresponding Laurent polynomial $\tilde{f} =\sum\limits_{m\in \Z^m} c_mx^m$.

A  toric variety $\cX$ with a collection of charts  $U_\alpha$ determines a system of cones $\{\sigma_\alpha\}$ in $\mathbb{R}^n$.
The variety $U$ can be recovered as the spectrum of this $\bK$-algebra $U = Spec \bK[S]$. We require that the set $S$ of "monomials" forms a basis of the space of regular functions on $U$. So,
we arrive at the fact that the ring $\bK[U]$ of regular functions on $U$ is the semigroup algebra $\bK[S]$ of a semigroup $S$ with coefficients in $\bK$. The variety $U$ is then  a spectrum of this $\bK$-algebra $U = \mathrm{Spec\,}\bK[S]$. Consider $S= \sigma \cap \Z^n$,  with  $\sigma$  being convex cone in $\R^n$ and being finitely generated. That is $\sigma$ must be polyhedral and rational. If $\sigma$ is generated by some basis of the lattice $\Z^n$ then  $U = \mathrm{Spec\,} K [\sigma \cap Z ^n]$ is isomorphic to $\A^n$.

 A toric variety is characterized by the fact that it contains an $n$-dimensional $T$ as an open subvariety, and the action of $T$ on itself by translations extends to an action on the whole variety. 

\begin{example}
For instance, to keep it simple take $K=\mathbb{C}$ and recall that the affine variety $(\Cc^{*})^n$ is a group under component-wise multiplication. A {\it torus} $T$ is an affine variety to $(\Cc^*)^n$, where $T$ inherits a group structure from the isomorphism. 

The character of a torus $T$ is a morphism $\chi:T\to \Cc^*$, that is a group homomorphism. The $n$-tuple $m=(a_1,\cdots, a_n)\in \Z^n$ gives a character $\chi^m:(\Cc^*)^n\to \Cc^*$ defined by 
\[\chi^m(t_1,\cdots,t_n)=t_1^{a_1}\cdots t_n^{a_n},\]
(this defines the Laurent monomials).

 The $\Cc$-linear space of all Laurent monomials is the ring $\Cc[t_1,t_1^{-1},\cdots,t_n,t_n^{-1}]$ of Laurent polynomials. The characters of $(\Cc^*)^n$ form a group isomorphic to $\Z^n$. Moreover, for an arbitrary torus $T$ its characters form a free abelian group $M$ of rank equal to the dimension of $T$. Furthermore, we have a 1-parameter subgroup $\Cc^* \to (\Cc^*)^n$  such that:
$t \mapsto (t^{a_1}, \cdots, t^{a_n}),$ where  $(a_1,\cdots, a_n)\in \Z^n.$

\end{example}

Let $\cV$ be a finite dimensional vector space over the field $\Q$ of rational numbers. A cone of $\cV$ is an intersection of a finite number of half-spaces. 
A subset of $\cV$  of the form $\lambda^{-1}(\mathbb{Q}_{+})$, where $\lambda: V \to \mathbb{Q}$ is a non-zero linear functional and $\mathbb{Q}_{+}$ denotes positive (or null) rational numbers, is called a half-space of
$\cV$ . A cone of $\cV$  is an intersection of a finite number of half-spaces. Note that cones are always convex, polyhedral, and rational. 
This leads to the following definition of an affine toric variety. Consider the point $(v_1,\cdots,v_k )\in \cV$. Then by $\langle v_1,\cdots,v_k \rangle$ we mean the smallest cone containing $(v_1,\cdots,v_k)$. Such a cone is said simplicial if the  $v_1,\cdots,v_k$ are linearly independent vectors.

The notion of  {\it lattice} is tightly related to the notion of toric varieties~\cite{She22}. By a lattice $\Lambda$ we mean a finite rank  free abelian group, the lattice $\Lambda^\star=\mathrm{Hom}(\Lambda,\Z)$ is called the dual  lattice of $\Lambda$. Thus a lattice of rank $n$ is isomorphic to $\Z^n$. 
By a cone in a lattice $\Lambda$ we mean a cone in the vector space $\Lambda _\Q =  \Lambda \otimes \Q$.  If $\sigma$ is a cone in the lattice $\Lambda$, then the intersection of the cone with the lattice $\sigma \cap \Lambda$ is a commutative subsemigroup in $\Lambda$.

\smallskip 

The affine scheme $Spec \bK[\sigma \cap  \Lambda]$, where $\Lambda$ is a lattice and $\sigma$ is a cone in $\Lambda$ is called  an {\it affine toric variety}.

General toric varieties are obtained by gluing together affine toric varieties; the scheme describing the gluing is given by a certain complex of cones, which we call a fan. 
\begin{dfn} $\cX$ is said to be (formally) toroidal at $x\in \cX$ if there exists a pair $(\Lambda, \sigma)$, where $\Lambda$ is a lattice and $\sigma$ a cone in $\Lambda$ with vertex, and a formal isomorphism of $(\cX, x)$ and $(X_\sigma, 0)$.  The toric variety $\cX_\sigma$  is called a local model for $\cX$ at $x$.
\end{dfn}

\smallskip 

To  each chart $U_\alpha$  in $\R^n$ corresponds a cone $\sigma_\alpha$ generated by the exponents of $ x_1^{\alpha}\cdots x_n^{\alpha}$ as Laurent polynomials in $x_1, …, x_n$. We regard an arbitrary Laurent polynomial $\tilde{f}$ as a rational function on $\cX$; as one sees easily, regularity of this function on the chart $U_\alpha$ is equivalent to $\text{Supp}(\tilde{f}) \subset  \sigma_\alpha$.  
So, the pair $(\cX, \{U_\alpha\})$ of the toric variety $\cX $ and the family of charts $ {U_\alpha}$ in $\R^n$, determines a system of cones $\{\sigma_{(\alpha)}\}$.
Therefore, in order to construct a toric variety it is sufficient to choose a system of cones $\{\sigma_{(\alpha)}\}$, under special requirements for example some system of dual cones  $\{\check\sigma_{(\alpha)}\}$ (this notion is convenient since it gives the opportunity to introduce the covariant character of the operations carried out in gluing together a variety $\cX$ out of affine pieces $U_\alpha$). 

The notion of general toric varieties is still related to affine charts $U_\alpha$ covering the variety $\cX$ with monomials. One must first of all, have a "monomial structure" on each affine chart $U_\alpha$; let us explain what this means, restricting ourselves to a single affine piece $U$.
Among the regular functions on an algebraic variety $U$ a certain subset $S$ of "monomials" is singled out. Since a product of monomials is again to be a monomial, we demand that $S$ is a semigroup under multiplication.

\section{Main theorem and its proof}~\label{S:3}
As a prelude to this section we start by recalling basic notions and properties of toric varieties (see section~\ref{S:3.1}). Once this is done we can proceed to the proof of our main statement (see section~\ref{S:3.2}).

\subsection{Lattices}\label{S:3.1}
Now, let us consider the exponential family of probability distribution defined by
\begin{equation}\label{E:exp1}
p( q;\theta) = p_0(\omega) \exp \{\theta^i q_i ( \omega) - \Psi ( \theta) \},
\end{equation}
with canonical distribution parameter $\theta=(\theta^1,\cdots,\theta^n) \in \R^n$, $\omega \in \Omega$, element of the sample space, the $q_i :\Omega \to \R$ are a family $q=\{q_i\}$ of random variables and $\Psi (\theta)$ is cumulant generating function.
The  $q_i(\omega), i \in I $ (some list of indices) is function defining directions of the coordinate axes, called  statistics (or directional sufficient statistics). This family generates a convex hull with $q_0 (\omega) = 1$. 

\begin{lem}
There exists a lattice $\Lambda$ generated by the $q(\omega)$. The group of translations $\Gamma$ acts on $\Lambda$. 
\end{lem}
\begin{proof}
The linear combination of these functions forms a linear vector space of dimension $n$. Denote it is as $\cX = \text{Span\ } \{1, q_1( \omega),, …, q_n (\omega)\}$. The space is decomposed into $\cX = 1 \oplus \cV$.
All the elements the space $\cV$ are in general positions, and this space contains whole simplex. The linear combination of integers valued functions of directional sufficient statistics
\[
q(\omega)= \alpha^0 q_0 (\omega) + \alpha^1 q_1(\omega)+  …. + \alpha^n q_n(\omega), \quad 0\leq   \alpha_i\leq 1,\qquad\sum\limits_{i= 0}^n\alpha^i  = 1.
\]
Then, the set $\Lambda = \{ q_0(\omega),   q_1(\omega),\cdots,  q_n(\omega)\}$ is a lattice of rank $n+1$ of integer valued directional functions. It is an integer lattice. 
Moreover, we know that the quotient $\cV/ \Lambda$ forms an $n$-dimensional torus in $\R^n$.
Now, consider a group of translations $\Gamma$. For each $q(\omega) \in \Lambda$ the  translation $\gamma_{q(\omega)} : \R^n \to \R^n $, given by $\gamma_{q(\omega)}(x) = x+ q(\omega)$ for $x\in \R^n$ is an automorphism of the real manifold. Since $\gamma_{q_1 +q_2}=  \gamma_{q_1} +  \gamma_ {q_2}$ the translation forms a group $\Gamma$, acting freely and discretely on $\R^n$. The set $z+ U,\, u \in U$ is neighborhood of $z$ and  $\gamma_ {q_1(z+ U)} \cap (z+ U)= 0$ for $q\ne 0\in \Lambda$.
Thus $\R^N / \Lambda$ is a compact manifold.
\end{proof}

\subsection{Toric and algebraic varieties}\label{S:3.2}

\vspace{5pt}
In the case of an exponential family defined on a finite sample space $\Omega$. Take $\Omega=\{\omega_1,\cdots,\omega_m\}$
and let us consider the directional statistics as the $n\times m$ matrix $\cQ=[q_{ij}]_{i=1,\cdots, n; \, j=1,\cdots,m}$, where $q_{ij} = q_i(\omega_j)$   and $q_0(\omega_j)=1$,  then 
\[
\exp\left\{\sum_{i=1}^n \theta^iq_{ij}\right\}=\prod_{i=1}^n t_i^{q_{ij}}=\boldsymbol{\tau_j}, \qquad  t_i=e^{\theta^i}\in \R_>.
\]
Let us assume that the $q_{ij}$ are integers. Then $\boldsymbol{\tau_j}$, for $j=1,\cdots,m$, are Laurent polynomials in $t_i$, which introduce a monomial parametrization~\cite{MaPi10}.

\smallskip 

We show that statistical manifolds can be considered as toric varieties. 

\begin{thm}
Let $\Omega$ be the sample space. If $\Omega$ is discrete then the exponential statistical pre-Frobenius manifold can be identified to a toric variety. If $\Omega$ is not discrete then the exponential statistical pre-Frobenius manifold can be identified to an algebraic variety. 
\end{thm}

\begin{proof}
Let us start with  $\Omega=\{ \omega_1,\cdots,\omega_m\}$ a finite, discrete sample space.
~

1. An exponential statistical pre-Frobenius manifold is given by the family $S=\{p(q,\theta)\}$ where $p(q,\theta)=\exp{\{\sum\theta^iq_i(\omega)-\Psi(\theta)\}}$ where $q_i(\omega)$, where $\omega\in\Omega.$

2. We use the construction of the family $S$ as a manifold, using the atlas $\{U_i,\phi_i\}_{i\in I}$. 
 
 3. For any point $P$ of the manifold $S$, there exists an $n$-vector space spanned by \vfill $\{1,q_1(\omega),\cdots,q_n(\omega)\}$, where $q_0=1$ and this follows form the normalisation. 
 
 4. Assume that $\Omega$ is discrete. Take the polynomial ring $\Q[{\bf y}]:=\Q[y_1,\cdots,y_n]$. Fix the subset $\cQ=\{\bf{q}_1,\cdots,\bf{q}_n\}$ where ${\bf q_i}=(q_{i1},\cdots,q_{im})^T$ of $\Z^m$  and the components $q_{ij}:=q_i(\omega_j)$,  with $\omega_j \in \Omega$. 
In other words we have a matrix $\cQ$ of size $m\times n$ with integer coefficients and columns given by the set $\{{\bf q_1},\cdots, {\bf q_n}\}$.

Note that each column ${\bf q_i}$ corresponds, for a given fixed event $\omega_j\in \Omega$, to a set of directional sufficient statistics. There exists $n$ such sets of directions.

\smallskip

5. Each vector ${\bf q_i}$ is identified with a monomial $t^{\bf q_i}$ in the Laurent polynomial ring $\Q[t^{\pm}]$, where $\Q[t^{\pm}]:=\Q[t_1,\cdots,t_m,t_1^{-1},\cdots,t_m^{-1}]$. More precisely, we know from the previous paragraph that $t_i={\exp{\theta_i}}$. 
So, the monomial $t_i^{q_i(\omega_j)}=\exp{\{\theta_iq_i(\omega_j)\}}$. Hence, the monomial $t_i^{\bf q_i(\omega_j)}$ can be interpreted as the probability of having the canonical parameter $\theta_i$ in the direction of $q_i(\omega_j)$ for the event $\omega_j.$  
 
In particular in the statistical manifold setting we have that this corresponds to $p(q;t)=t_1^{{\bf q}_1}\cdots t_n^{{\bf q}_n}$.

Now, we apply the following construction. First, take the semigroup homomorphism:
 \[\pi: \N^n\to \Z^m,\quad \bu=(u_1,\cdots,u_n)\mapsto\sum_{i=1}^n u_i{\bf q_i}.\] 
 The image of $\pi$ is the semigroup:
 \[\N\cQ=\{\lambda_1{\bf q_1}+\cdots \lambda_n{\bf q_n}\, :\, \lambda_1,\cdots, \lambda_n\in \N\}\]

This map $\pi$ lifts to a homomorphism of semigroup algebras:
\[\hat{\pi}: \Q[{\bf y}]\to \Q[t^{\pm1}],\quad y_i\mapsto t^{\bf q_i}.\]

Now, to every point $P\in U_i$, where $(U_i,\phi_i)$ is a chart, we apply this homomorphism construction. 
The kernel of $\hat{\pi}$ is the toric ideal $\cI_{T}$ of $\cQ$ and its corresponding affine variety is irreducible. It is the Zariski closure of the set of points $(t^{\bf q_1}, \cdots, t^{\bf q_n})$ where $t\in (\Q^*)^m$. The multiplicative group $(\Q^*)^m$ is known as the $m$-dimensional algebraic torus. The variety of the form $V(\cI_{T})$ is the affine toric variety 


The kernel $\textrm{Ker}(\pi)$ amounts to solving in $\N^n$ a system of equations, where coefficients are the directional sufficient statistics. This corresponds to a linear regression.


To be more precise, the toric ideal $\cI_T$ is spanned as a $\Q$-vector space by the set of binomials: 
\[\{\, {\bf y}^\bu-{\bf y}^\bv\, :\, \bu,\bv \in \N^n\quad  \text{with}\quad \pi(\bu)=\pi(\bv)\, \}.\]

In particular, each polynomial in $\cI_{\cA}$ is a $\Q$-linear combination of these binomials. 

Moreover, since every vector $\bu\in \N^n$ can be written uniquely using the decomposition $\bu=\bu^+-\bu^-$, where $\bu^+$ and $\bu^-$ are non-negative and have disjoint support, the $\text{Ker}(\pi)$ for the sublattice of $\Z^n$ consists of all vector such that $\pi(\bu^+)=\pi(\bu^-)$. 
So, the previous statement means that $\cI_T$ is generated by ${\bf y}^{\bu^+}-{\bf y}^{\bu^-}$ and $\bu\in \text{Ker}(\pi).$

This construction holds for any point $P$ in any chart $(U_i,\phi_i)$ of the manifold of probability distributions. 
The coordinate functions $y_i$ on the chart $U_i$ can be expressed as Laurent monomials in the adequate coordinates.  
In changing from one chart to another the coordinate transformation is monomial. 
This forms a smooth toric variety, where we have a collection of charts $y_i: U_i\to\Q^n$, such that on the intersections of $U_i$ with $U_j$ the coordinates $y_i$ must be Laurent monomials in the $y_j$.

Assuming that $\Omega$ is not discrete, the construction is essentially the same. But we cannot have a lattice and thus this does not give a toric variety. It simply gives an algebraic variety. 
\end{proof}

\section{Web designs}~\label{S:2}
In \cite{CoMaMaA,CoMaMaB} the utility of web theory for models of databases subject to noise (i.e. spaces of probability distributions on a finite set)
turned out to be a successful tool for unravelling hidden symmetries of these spaces. For instance, this allowed to show the existence of mathematical structures describing symmetries of relevant geometries where Commutative Moufang Loops (CML) play a prominent role. 

Recall that, independently, CML relate strongly to algebraic geometry. More precisely, there is a tight relation between Moufang loops and the set of algebraic points of a smooth cubic curve in a projective plane $\mathbb{P}^2_K$ defined over a field $K$. Indeed, the set $E$ of $K$-points of such a curve $X$ forms a CML where the composition law is $x\circ y = u\star (x\star y)$, if $u+x+y$ is the intersection cycle of $X$ with a projective line $\mathbb{P}^1_K \subset \mathbb{P}^2_K$ (see~\cite{Man86}, Chp.1) 

This remark motivates our further investigations in this paper. Indeed, taking into account the previous section, it was proven that certain models related to transmitting information have the structure of an algebraic variety (toric variety). Using the results presented in the context of information theory and Commutative Moufang loops from~\cite{CoMaMaA,CoMaMaB} we can relate this to problems concerning the development of {\it new relations} between algebraic geometry, information theory and web theory. This development reinforcing thus the philosophy of algebraic statistics (i.e. motivated by explaining statistics by means of algebraic tools) and results therein. 

It is rarely the case that algebraic techniques are ready-made to address statistical challenges. Usually new algebraic results in this domain need to be developed. This section is precisely about proving and illustrating that our webs are algebraizable (as well as hexagonal and isoclinic). Since statistical data are encapsulated in the web structure, the property of having algebraizable webs provides us with a powerful algebraic tool for addressing statistical problems.

The importance of considering algebraic webs comes from the following facts. First, algebraic webs are closely related to loops and quasigroups. 
The left loops (loops with a left type of multiplication) are in bijection with left homogeneous spaces (which are equipped also with a left type of multiplication). Therefore, we can state that left loops and  left homogeneous spaces describe the same structure on a given affine connection manifold.

Now, consider an affine connection manifold (see \cite{No} for the definition). The space of probability distributions on a finite set is precisely an affine connection manifold~\cite{CoMa20}.
For these types of manifolds, one can introduce a specific type of local loop in the neighbourhood of any point. For details concerning the properties of the considered type of loop we refer to \cite{Kika}. These loops are uniquely determined by means of the parallel transport along a geodesic (see \cite{Sa1,Sa2} for developments on this topic). Then, this family of local loops defines a covering of the manifold and uniquely determines all affine connections. Hence, the application of algebraic webs provides us with important information concerning the geometrical structure of the manifold under investigation.

 \smallskip 
 
 \subsection{Definition of webs}
Consider $(x_1,\cdots, x_n)$ local coordinates on a smooth $n$-manifold $M$ and a foliation of dimension $n-k$, or codimension $k$, has leaves defined by equations of the form \[f_1(x)=const,\cdots, f_k(x)=const,\quad x=(x_1,\cdots, x_n),\]
where $f_i$ are smooth functions and their Jacobian matrices are of rank $k$ everywhere. The functions $f_1,\cdots, f_k$ are defined up to an arbitrary smooth transformation. A $d$-web consists of $d$ foliations where the leaves are everywhere in general position. 

More concretely for $\cD$ an open domain of a differentiable manifold $M^{nr}$ of dimension $nr$, a $d$-web will be denoted by the symbol $W(d,n,r)$ (or $W$ when there is no possible confusion) where:
\begin{itemize}
\item the number of foliations is $d$;
\item the codimension of the foliations is $r$;
\item the dimension of the smooth manifold is $nr$.
\end{itemize}

For simplicity, we will call a web $k$-dimensional whenever its foliations are of dimension $k$.

Given a web, it is natural to ask whether a notion of equivalent webs exists. In particular, for a given web with $d$ foliations given by \[f_i(x_1,\cdots , x_n)=\text{const}\quad 1\geq i\geq d,\]
where we suppose that $f_i$ are smooth with non-null gradient. Then, the function $f_i$ can be replaced by the function $g_i(f_i)$ such that the gradient of $g_i$ is non-null, without changing the $i$-th foliation.

Webs can be hexagonal, algebraizable, isoclinic, parallelizable, Grassmanniazable. We discuss these definitions below, since  we will prove that, on statistical Frobenius manifolds,  the webs satisfy those properties. 

\subsection{Hexagonal and parallelizable webs} 
The class of webs being {\it hexagonal} forms an important class. Figure~\ref{F:hex} presents an example of a hexagonal webs. Let $\cD$ be a two dimensional domain. Take a point $p \in \cD$ and consider three regular families of smooth curves in $\cD$, being in general position. In the neighbourhood of that point $p$, where the web is given, one can construct a family of hexagonal figures. However, the problem occurring in certain cases is that the hexagon might not be closed (see figure Figure 1 in~\cite{Go88}) and in this case the web cannot be hexagonal. 

Note that the case of families of three webs formed each by a set of three parallel lines gives exactly a (closed) hexagon. This specific example where the three web are formed by three families of parallel straight lines is called {\it parallelizable}. A 3-web equivalent to a parallel web is called parallelizable.

\begin{exa} 
One can consider in a $2r$-dimensional affine space $\A^{2r}$, three families of parallel $r$-dimensional planes, which are in general position. They form a three-web, which is called a {\it parallel} three-web. 
\end{exa}

Furthermore a useful statement is that: 

\begin{thm}\label{T:Ash}[Thm.1.5.2 and Cor. 1.5.4~\cite{Go88}]
A web $W(n+1,n,r)$ for $n>2$ is parallelizable if and only if it is torsionless. 
Moroever, given an $(n+1)$-web being parallelizable, all its $(k +1)$-subwebs are also parallelizable.

\end{thm}

We prove the following statement. 
\begin{prop}
Consider the web $W=W(2,3,r)$, with $r\geq 1$ on the statistical Frobenius manifold. Then, the web $W$ is hexagonal.  
\end{prop}

\begin{proof}
 Let $W$ be the web of the statistical Frobenius manifold. Having the Frobenius condition implies that the Chern connection is flat and that the torsion vanishes. Thus, it fulfils the hexagonality property of the web.
\end{proof}

Moreover, we are able to show that a class of  webs of type $W(n+1,n,r)$ on the statistical Frobenius manifold is parallelizable.

\begin{prop}\label{T:para}
Suppose that we have a web $W(n+1,n,r)$, with $n>2$ on the statistical Frobenius manifold. Then $W$ is parallelizable.
\end{prop} 
\begin{proof}
Indeed, for a Frobenius manifold the Chern connection is flat and torsionless. Applying the theorem~\ref{T:Ash} the web is parallelizable.
\end{proof}

\medskip
\subsection{Foliations for the multinomial distribution}

Let $(\Omega,{\mathcal P}(\Omega),P)$ be a probability space, where $\Omega =\{\omega_0,x_1,\dots,\omega_d\}$ a finite
space of cardinality $\#\Omega =d+1$ and ${\mathcal P}(\Omega)$ the algebra of all measurable
subsets of $\Omega$, $\#{\mathcal P}(\Omega)=2^{n+1}$. This algebra is generated by the family
$\{\{\omega_0\},\{\omega_1\},\dots \{\omega_d\}\}$  of the $d+1$ one point subsets of $\Omega$. The finite
probability $P$, absolutely continuous with respect to the uniform measure is well
 defined by its distribution function:
\begin{equation}
p^i=P[\{\omega_i\}], \qquad
 0< p^i<1,\qquad \sum\limits_{i=0}^d p^i=1.
\end{equation}

The distribution $p$ appears as a random variable $p:X \to
[0,1]$ which can be written:
\begin{equation}
P=\sum_{i=0}^d p^i\delta_{\omega_i},\quad\text{ with } \quad0< p^i<1,\quad \sum\limits_{i=0}^d p^i=1.
 \end{equation}
where\[\delta_{\omega_k}(\omega)= \begin{cases}
1&\text{ if  $\omega=\omega_k$},\\ 0& \text {otherwise}.
\end{cases}
\]

\subsection*{The 2-dimensional case}

In what follows we recall a construction, which is naturally generalised to any dimension $n$.
Let,  $(\Omega_3, \A_3)$ be a measurable space where $\Omega_3=\{\omega_{1},\omega_{2},\omega_{3}\}$ and $\A_3$  the algebra of the subset of $\Omega$, of cardinality $\# \A_3=2^{3}$ (the algebra of events). 

Any probability distribution $P$ on the algebra $\A$ is given by a triplet $(p^{1},p^{2},p^{3})$ such that \[ p^{j}=P[\{\omega_{j}\}],\quad j=1, 2, 3\] (the probability of the event (atom) $\{\omega_{j}\}$ of $\A_3$). 

In the following we associate a probability distribution $P$ with a point of the simplex
\begin{equation}\label{E:11.1}
P\leftrightarrow \mathbf{p} \in \left\{(p^{1},p^2,p^3) \mid  \ p^{1}+p^2+p^3=1\ ;\ \forall \, j,\ p^{j}\geq 0 \right\}.
\end{equation}

\subsection*{Embedding in $\R^3$}

Let $\vec e_1=(1,0,0), \vec e_2=(0,1,0), \vec e_3=(0,0,1)$ be the canonical basis of the affine space $\R^3$ where we fixed the origin at $(0,0,0)$.

The probability distribution concentrated on the elementary events $\{\omega_{k}\}, \ P[\{\omega_k\}]=1$, can be identified with the basis vectors
\begin{equation}
\vec e_{k}=(e_{k}^{1},e_{k}^{2}, e_{k}^{3}), \text{ with } e_{k}^{j}=\delta_{k}^{j}.
\end{equation}

The family $\{\vec e_{k}\}_{k=1}^{3}$ defines a canonical basis of $\mathbb{R}^{3}$ ($\mathbf{e}_{i} \cdot \mathbf{e}_{j} = \delta_{i}^{j}$ where $\cdot$ denotes the scalar product in $\mathbb{R}^{3}$).
 
Then the probability $P$ is in bijection with the vector:
\begin{equation}\label{E:vectp}
\vec p= p^{1}\vec e_1+ p^2\vec e_2+p^3\vec e_3, \text{ where } p^{j}\geq 0,\ p^1+p^2+p^3=1.
\end{equation}

\begin{center}

\begin{tikzpicture}[scale=1.0]
 \draw(0.12,-0.2) node{$\mathbf{0}$};
 
 \draw [black,dashed,->] (0, 0) -- (0,2.9);
 \draw[black,dashed,->]  (0, 0) -- (3.0,-0.43);
\draw [black,dashed,->] (0, 0) -- (-1.98,-1.0);

\draw [black] (0, 3) -- (-2,-1);
\draw [black] (0, 3) -- (3.2,-0.45);
\draw [black] (-2, -1) -- (3.3,-0.45);

\draw(0.4,3.2) node{$\be_{3}=(0,0,1)$};
\draw(4.2,-0.5) node{$\be_{2}=(0,1,0)$};
\draw (-1.6,-1.2)node{$\be_{1}=(1,0,0)$};

\draw(0.2,1.5) node{$\vec{e}_{3}$};
\draw(1.2,0.1) node{$\vec{e}_{2}$};
\draw (-0.8,-0.1) node{$\vec{e}_{1}$};
\end{tikzpicture}
\end{center}

Any point $\bp$ of the simplex being a probability $P$ such that $P[\{\omega_{j}\}]=p^{j}$, is in bijection with the  point $\vec p$ of the space $(\mathbb{R}^{3}, (0,0,0))$ and verifying~\eqref{E:vectp}.

\smallskip 

A point $\bp$ lies on the simplex face with vertices $\mathbf{e}_{(j_{1})},\dots,\mathbf{e}_{(j_{k})}$, if and only if $P[\{\omega_{j(1)}\}]+\dots+P[\{\omega_{j(k)}\}]=1$. Thus, the face of the simplex forms the collection $\mathrm{Cap}(\Omega',\A')$ of probability distributions on the finite algebra $\A'$, where $\Omega'=\Omega-\cup_{i=1}^{k}\{\omega_{j(i)}\}$ and $\A'$ the algebra of subsets of $\Omega'$.

This construction is reminiscent of problems in planar geometry. In particular, the Ceva and Menelaus theorems.  
The Cevian (Ceva line) is a line starting from a vertex and connecting the opposite face. The following construction allows a description in terms of a continuous one parameter groups. Using this construction one can introduce the notion of derivation on the simplex.

\subsection{An invariant basis, for the $n$-dimensional case}
The construction above works for any dimension.
Let us introduce the following families of vectors.

\begin{itemize}
\item The family $\{\vec y_{k}(\bp)\}_{k=1}^{n}$ of non independent ``vectors'' in $\mathbb{R}^{n}$ starting at the point  $\bp$ and ending at vertex $\be_{k}$. This family of vectors lie on the simplex:
\begin{equation}\label{E:Y1}
\vec y_{k}(\bp)= \vec e_{k}-\vec p,\ k=1,\dots,n
\end{equation}
\smallskip 
\item The family $\{\vec z_{k}(\bp)\}_{k=1}^n$ of non independent ``vectors'' in $\mathbb{R}^{n}$. This starts at the point $\bq_k$, indexed by the vector $\vec q_{k}$, where the Cevian intersected the $k$-face,  opposite to the vertex from which is is issued. These vectors lie on the simplex:  \begin{equation}\label{E:Z1}
 \vec z_{k}=(1-p^{k})^{-1}\vec y_{k}(\bp)=\vec e_{k}- \vec q_{k},
\end{equation}
where the vector 
\begin{equation} \vec{q}_{k}=\sum_{i\ne k}q_{k}^{i}\vec{e}_{i},\text{ with } q_{k}^{i}= \frac{p^{i}}{1-p^{k}},
\end{equation}
corresponds to a conditional probability distribution $P[\ \cdot \mid \Omega\setminus \{\omega_{k}\}]$, lying on the $k$-th face of the simplex (for the 2-dimensional case: the edge).
\end{itemize}
\begin{rem}This construction holds naturally for $n\geq 3$.
\end{rem}
From the definition~\eqref{E:Y1} of  $\mathbf{y}_{k}(\mathbf{p})$ we can as well, introduce the  following generating systems of vector fields on the simplex:
 \begin{equation} \label{E:X1}
\vec x_{k}(\bp)=p^{k}\vec y_k.
\end{equation}

This generating system verifies the condition
\begin{equation} \label{E:X2}
\sum_{k=1}^{n}\vec x_{k}(\bp)=0.
\end{equation}

A point $\mathbf{p}$ in the interior of the simplex can be written as:
 \begin{equation}\begin{aligned}
 \vec p &= p^{k}\vec e_{k}+ (1-p^{k})\sum_{i\ne k}q^{i}_{k}\vec e_{i}, \text{ with } 0<p^{k}<1,\\
 &=p^{k}\vec e_{k}+(1-p^{k})\vec q_{k},
 \end{aligned}
 \end{equation}
and this  for each choice of the vertex $\vec e_{k}$.

\begin{center}
\begin{tikzpicture}[scale=1.5]
 \draw(0.1,0) node{$\mathbf{0}$};
 \draw [black,dashed] (0, 0) -- (0,3);
 \draw(0,3.1) node{$\mathbf{e}_{3}$};
\draw[black,dashed]  (0, 0) -- (3.5,-1.5);
\draw(3.5,-1.5) node{$\mathbf{e}_{2}$};
\draw [black,dashed] (0, 0) -- (-3.5,-0.5);
\draw(-3.5,-0.5) node{$\mathbf{e}_{1}$};
\draw [black] (0, 3) -- (-3.3,-0.46);
\draw [black] (0, 3) -- (3.3,-1.45);
\draw [black] (3.3, -1.45) -- (-3.3,-0.46);
\draw[black]  (0, 3) -- (-1.45,-0.75);
\draw[black,dashed]  (0, 0) -- (-1.5,-0.75);
\draw(-1.5,-0.85) node{$\mathbf{q}_{3}$};
\draw[black,dashed]  (0, 1.6) -- (-0.7,1.2);
\draw[black,very thick,dashed,->]  (0, 0) -- (-0.7,1.2);
\draw(-0.8,1.2) node{$\mathbf{p}$};
\draw(0.15,0.9) node{$p^{3}$};
\draw(0.2,2.3) node{$1-p^{3}$};
\draw(-0.6,0.6) node{$\vec p$};
\draw[black,very thick,->]  (-0.7,1.2) -- (-0.31,2.2);
\draw(-0.8,1.7) node{$\vec x_{3}(p)$};
\end{tikzpicture}
\end{center}

Then the independent vectors $ \vec x_{1}(\bp),\vec x_{2}(\bp),\dots, \vec x_{n}(\bp)$ define an invariant generating system under the permutation of the vertices. Moreover, it allows to define a barycentric coordinate system.
\begin{rem}
Note that this barycentric coordinate system is just the coordinate system induced by the three Ceva lines intersecting at $\bp$.
\end{rem}

\subsection{Change of coordinates}
In the case of a manifold $S$ of multinomial distributions we can use  as coordinate systems $\theta=(\theta^1,\dots,\theta^n)$ such that:  

\[\theta^1=p_1,\dots, \theta^n=p_n,\theta^{n+1}=p_{n+1}\]
and satisfying the condition \[\theta^1+\dots+\theta^{n+1}=1.\] 

As for the tangent space, it has $n$ generators $\partial_1,...,\partial_n$ where $\partial_il(x,\theta)=\frac{\delta_{x_i}}{\theta^i}-\frac{\delta_{x_{n-1}}}{\theta^{n+1}}$. 
 
In the $\theta^{k}=p^{k}, \ k=0, \dots, n$ coordinate system this space can be realized as an open simplex whose the vertices are given by $\mathbf{e}_{k}=\underbrace{(0,\dots,0,\underset{k}1,0,\dots,0)}_{d+1}$.
An interesting remark concerns the change of coordinate system. Let us define $(\eta^i)_{i=1}^{n+1}$ by \[\eta^1=2\sqrt{p_1},\dots, \eta^{n+1}=2\sqrt{p_{n+1}},\]
so that \[\sum_{i=1}^{n+1} (\eta^i)^2=4\] which is exactly the equation of a sphere of radius 2. This interesting change of coordinates presents the manifold $S$ as a part of the $n$-dimensional sphere with radius 2, embedded in the $(n+1)$-Euclidean space with coordinate system $(\eta^1,...,\eta^n)$.

\subsection{Parallelizable webs for the case of  multinomial distribution}
A concrete geometric discussion of the application of our previous statements for the case of the multinomial distribution is presented. Consider the statistical manifold corresponding to an $n-$simplex. Since we consider only those $(\theta)_{i=1}^n$ such that $\sum_{i=1}^n \theta_i=1$, we are thus investigating $n$ webs on a given $(n-1)$-face of the simplex. 

\smallskip
We give a geometric proof of parallelizable webs for the pre-Frobenius statistical model, corresponding to the case of multinomial distribution. This construction relies on the Cevians.
\begin{prop}
Let us consider the $n$-dimensional pre-Frobenius statistical model, corresponding to the case of multinomial distribution. Then, the $(n+1)-$web, i.e. given by  Cevians form a parallelizable web.
\end{prop}

The proof is by induction, presented in the geometric construction below.

\subsection*{Proof}
Let ${\bf e(1), e(2),... ,e(n+1)}$ be the vertices of the simplex. Since we consider only the $(n-1)$-face corresponding to $\sum_{i=1}^n \theta_i=1$, we have $n$ vertices. We construct the $n$-web using the Cevians (i.e. lines that connect a given vertex to its opposite face). 

\subsection*{Low dimensional case}
Let us first focus on the low dimensional case, i.e. the 3-simplex. For simplicity, let us call the vertices of the 2-face (triangle) we are interested in $A,B,C$. The Ceva relation (given by the Ceva and Menelaus theorem) states that we have the following relation:

\begin{equation}\label{E:Ceva1}\frac{A'B}{A'C}\times\frac{B'C}{B'A}\times\frac{C'A}{C'B}=-1.\end{equation} 
This relations holds for a triple of lines $(AA'), (BB')$ and $(CC')$ which are either concurrent in one point $k$ or parallel.

\smallskip 

In the first case, one draws lines from each vertex to its opposite edge, being concurrent at one point $k$ and satisfying the relation~\ref{E:Ceva1}. The second, is obtained by drawing three parallel lines, where each line contains one unique vertex of the triangle and the relation~\ref{E:Ceva1} is satisfied. Following the definition of a parallelizable 3-web and restricting our attention to the case where the manifold is $\alpha$-flat (that means Frobenius) we can state the following.  
 
\begin{lem}
Consider the 2-dimensional statistical pre-Frobenius manifold, given by the multinomial distribution.
Then, the three-web can be given by the Cevians and this web is parallelizable.  
\end{lem}

\begin{proof}
It is easy to verify that the construction above of Cevians on the 2-face of a 3-simplex corresponds to the definition of a three web. Applying the theorem of Ceva, we can draw those Cevians as parallel lines.  This happens to be exactly the definition of a parallelizable web.    
\end{proof}

\subsection*{Generalisation to higher dimensions}
We generalise this statement to higher dimensions. We adopt this construction for the $n$-dimensional case i.e. for the case of an $n$-simplex. 
Let $A_1, \dots , A_{n+1}$ be the vertices of an $n$-simplex. Let $B_{ij}$ be the points lying on 1-dimensional edges $A_iA_j$ of the $n$-simplex, which do not coincide with the vertices of the simplex, such that $1\leq i\leq n$ and $i+1 \leq j\leq n+1$. 

\smallskip

The theorem is given as follows. 
\begin{thm}[Generalised Ceva] 

Let us consider the hyperplanes

 $A_{i_{1}}\dots A_{i_{n-1}}B_{km}$, where the indexes are ordered as follows $i_1 <...<i_{n-1}$, $k<m$ and $i_j\neq k,m$ lie in the set $\{1,...,n+1\}$. Then, these hyperplanes  have a common point if and only if the following equalities are fulfilled:

\begin{equation}\label{E:Ceva}\frac{A_iB_{ij}}{B_{ij}A_j}\times\frac{A_jB_{jk}}{B_{jk}A_k}\times\frac{A_kB_{ik}}{B_{ik}A_i}=1\end{equation}
for $i=1,...,n-1,$ $j=2,...,n,$, $k=3,...,n+1$, and $i<j<k.$\end{thm}

Any hyperplanes verifying Ceva's relation have a common point. However, we can use the version where Cevians are parallel, thus creating a parallelizable $n$-web for the manifold corresponding to the interior of this $n$-dimensional face. 

By induction, we can construct a parallelizable $(n+1)$-web for this $n+1$-dimensional face, based on the previous step. 

Therefore, we have a new proof concerning the parallelization of webs (i.e Prop.~\ref{T:para}) via the geometric construction above, relying on Cevians.

\subsection{Webs for manifolds over algebras}
In this section we consider manifolds over real (finite dimensional, unital) algebras. The theory of generalised Cauchy--Riemann equations was developed by Scheffers~\cite{Sch93} and allows a rigorous definition of analytic manifolds over finite dimensional algebras. Recollections on the construction of these manifolds (which relies on real realisations of modules over those algebra) are omitted. However, we refer to~\cite{CoMa20} for a full exposition on this topic.

\subsubsection{}

Consider a rank $r$ algebra $\fA_r=\{\be_1,\cdots, \be_r\}$ over the real numbers, where the multiplication is given by $\be_i\cdot\be_j=c_{ij}^k\be_k$. It is a unital algebra equipped with the commutativity and associativity conditions, given respectively by the formulas:  
\begin{equation}\label{eq:a}
c^k_{ij}=c^k_{ji}, \quad
c_{ik}^sc_{sk}^m = c_{ik}^s c_{sj}^m.
 \end{equation} 
Any element $\bA$ of the algebra $\fA_{r}$ is given by the linear combination \[\bA=A^i\be_i,\] where $A^i$ is a (real) number and $\be_i$ are generators of $ \fA_r$ and the principal unit $\epsilon$ is given by $\epsilon^i\be_i$.

Consider a manifold over the algebra $ \fA_r$. Denote it $\frak{M}( \fA_r)$. We shortly recall the key step to its construction: the generalized Cauchy--Riemann equations, allowing a rigorous definition of manifolds defined over such algebras. Finally, we discuss the existence of 3-webs on these manifolds.

\smallskip 

Let $y=f(x)$ be a (analytic) function, whose domain and range belong to a commutative algebra (i.e. $C_{jk}^h=C_{kj}^h$). We put $x=\sum_ix_ie_i,$ $y=\sum_iy_ie_i.$
From the generalized Cauchy--Riemann we have the following:
\[\sum_h\frac{\partial y_i}{\partial x_h}C_{jk}^h=\sum_h\frac{\partial y_h}{\partial x_i}C_{hk}^j,\]
where $C_{jk}^h$ are the constant structures. 
 
 \smallskip 
 
The existence of these functions allow us to define a {\it differentiable $n$-manifold} over the algebra $ \fA_r$. Note that in local coordinates, if the rank of the algebra $ \fA_r$ is $r$ (i.e. $rk( \fA_r)=r$), then the local coordinates point $p$ on the manifold $\frak{M}( \fA_r)$ are given by an $nr$-tuple of elements in this algebra: 
\[(\, x_{1}\be_1,\cdots , x_{n}\be_1,\cdots, x_{1}\be_r,\cdots , x_{n}\be_r)\,,\]
where $\{ \be_i\}_{i=1}^r$ are generators of $ \fA_r$ and $(x_i)_{i=1}^n$ are the coordinates of the manifold $\fM$.

For instance a two-dimensional manifold $\frak{M}$ over a rank $r$ algebra $ \fA_r$ (commutative, associative, unital) admits a real realization in the form of a real $(2r)$-dimensional manifold $M$. The admissible coordinate transformations are determined by functions satisfying generalized Cauchy--Riemann equations (i.e. compatible with the algebra $\fA_r$).

\begin{rem}
Our investigations will rely on the example cited above. Indeed, from the works~\cite{CoCoNen3} it follows that the statistical Frobenius manifold can be considered as a a projective manifold defined over  a (split) real algebra of rank 2, generated by a pair of idempotents. 
\end{rem}

\smallskip 
\subsubsection{}
Three webs can naturally be defined on these types of manifolds. Let $\cD$ be an open subspace defined on a 2-dimensional manifold defined over $  \fA_r$. Given a differentiable function $f:\cD\to   \fA_r$ (with values in $  \fA_r$ and defined in a domain $\cD$), the foliations of the web can be defined by the following type of equations:

\[x=a,\, y=b, \quad  f(x,y)=c, \quad a,b,c\in   \fA_r, \]

where we suppose that the partial derivatives of $f(x,y)$ are non-nul (and that they are not zero in the algebra $ \fA_r$).
The structure equations for a three web on $\frak{M}( \fA_r)$ has the same form as the structure equations for any two dimensional three web in the real plane. 

This leads to the following conclusion.
\begin{lem}
Let $  \fA_r$ be an algebra of rank $r$. Consider a manifold of dimension 2 on the rank $r$ algebra $  \fA_r$. Let $W_{\frak{M}}$  be a $d$-web on $\frak{M}(  \fA_r)$. Then, the foliations $\lambda_\alpha^{\frak{M}}$ where $\alpha =1,\cdots, d$ are of dimension $r$.\end{lem}

\begin{proof}
By hypothesis the manifold $\fM$ is of dimension 2. So, the realisation of the manifold over the algebra $\fA_r$ gives a real manifold of dimension $2r$.  Realisations of the web foliations are given, for any $\alpha=1,\cdots, d$ by the $r$-tuple $\lambda_\alpha=(\lambda_\alpha^{\frak{M}}\be_1, \cdots, \lambda_\alpha^{\frak{M}}\be_r)$, where 
 the $\lambda_\alpha^{\frak{M}}$ are one-dimensional foliations coming from the web $W_{\frak{M}}$ on $\frak{M}$.\end{proof}

Following~\cite{CoMa20,CoCoNen1,CoCoNen2,CoCoNen3}, the specific case under consideration implies that $\fA_2$ is an algebra of rank two, generated by a pair of orthogonal idempotents. 

It is generated by $\langle \be_1,\, \be_2 \rangle$, where $\be_i\be_j=c_{ij}^k\be_k$. 
Under a suitable change of basis, we can rewrite this pair in a canonical basis $\langle \be_+, \be_-\rangle $, where $\be_+^2=\be_+$ (resp. $\be_-^2=\be_-$) and $\be_+\be_-=\be_-\be_+=0.$ In other words this is a pair of orthogonal idempotents.

\smallskip 

Let us recall that by definition (Def. 1.11.9~\cite{Go88}), a web $W(n + 1,n,r)$ with vanishing skew-symmetric part of its torsion tensor is called {\it isoclinicly geodesic}.

By~\cite{Ti75}, let us recall the following result.
\begin{lem}\label{T:iso}
A three-web $W = (\frak{M}(  \fA_r), \lambda_\alpha)$, being a real realization of a two-dimensional three-web over $ \fA_r$ (a commutative, associative and unitary algebra), is an isoclinicly geodesic three-web. \end{lem}

\subsubsection{Subwebs and ideals of $\fA_r$}

The pair of orthogonal idempotents $\be_+$ and $\be_-$ in the split rank 2 algebra $\fA_2$ generate a pair of ideals. Due to the existence of ideals in the algebra $ \fA_2$, subwebs emerge. 
 
 \begin{prop}
 Let $\fA$ be a rank two split algebra, given by a pair of orthogonal idempotents $\be_+$ and $\be_-$ generating respectively the ideals $\cI$ and $\cI^*$. Then, the statistical Frobenius manifold, corresponding to module over $\fA$, is endowed with a pair of subwebs, lying respectively in the submanifolds, corresponding to the realisations of the modules $\fM$ over the ideals $\cI$ (resp. $\cI^*$). 
 \end{prop}

 \begin{proof}
 Let $\fA=\cI\oplus \cI^*$, where the ideals $\cI$ and $\cI^*$ are of dimension $m=1$.
The algebra is of rank $r=2$ and the module $\fM$ is of dimension $n$. Then, by applying the previous statements, we have the existence of a web $W$ on $\frak{M}(\fA)$. 
The fibrations $\lambda_\alpha$ of $W$ are $(2n)$-dimensional. As a general result, it is known that for a reducible algebra $ \fA_2=\cI\oplus \cI^*$, where the ideals of respectively of rank $m$ and $r-m$ there exists a double fibration of the (real realization of a) web $W$ into $(nm)$-dimensional and $n(r - m)$-dimensional subwebs. These subwebs are real realizations of webs over the algebras $\cI$ and $\cI^*$.

Now, applying this result to our context, the (sub)webs lie on the (sub)manifolds given by the realisations of the module $\fM(\cI)$ (resp. $\fM(\cI^*)$) over the ideal $\cI$ (resp. $\cI^*$).  Explicitly, one can check this statement by hand with the following formula:
 \[\lambda_\alpha=(\lambda_\alpha^{\frak{M}}\be_+, \lambda_\alpha^{\frak{M}}\be_-)\] where
 $\be_+$ (resp. $\be_-$ )  generates the ideal $\cI$ (resp. $\cI^*$) and $\alpha\in \{1,2,\cdots, d\}.$
 
Since $\cI$ and $\cI^*$ are one dimensional, we can say that the $2n$-dimensional web $W$ admits a pair of $n$-dimensional subwebs, lying on the submanifolds corresponding to the realisations of modules over the ideals $\cI$ (resp. $\cI^*$).
\end{proof}

\begin{prop}
Let $\fA$ be a rank two split algebra, given by a pair of orthogonal idempotents $\be_+$ and $\be_-$ generating respectively the ideals $\cI$ and $\cI^*$. 
 Then, the statistical Frobenius manifold, corresponding to a module over $\fA$, is endowed with a pair of subwebs being parallelizable. 
\end{prop}

\begin{proof}
A web $W$ is parallelizable if and only if its torsion and curvature tensors are equal to zero. For Frobenius manifolds the Chern connection is flat and the torsion is null. So, $W$ is parallelIizable. In virtue of Proposition~\ref{T:Ash} its subwebs are also parallelizable. 
\end{proof}

\smallskip

\begin{thm}
Let us consider the  statistical Frobenius manifold from the point of view of a module $\fM$ over a split algebra of rank 2. Suppose that $\fM$ is of dimension 2. Then, the three web $W(3,2,2)$ is an isoclinic Moufang web. The three-web is also an isoclinic group web.
\end{thm}   
\begin{proof}
A three-web $W$ is a group three-web if and only if its curvature tensor is equal to zero (Theorem 1.11 \cite{AkSh92}). This is satisfied because we have a Frobenius manifold. Furthermore, it is hexagonal (because of the Frobenius condition). By theorem. 3.5~\cite{AkSh92} it is thus transversally geodesic and all its 2-dimensional subwebs are hexagonal. Therefore, we have gathered the necessary conditions to apply lemma~\ref{T:iso}, and deduce that the three web is an isoclinicly geodesic three web. 

Now, theorem 2.6.4~\cite{Go88} states that a web $W(n+1,n,r)$, for $n \geq 2, r \geq 2$, is Grassmannisable if and only
if it is both transversally geodesic and isoclinic. So, we deduce that we have indeed a three-web being Grassmannisable. 

By Theorem 4.12\, \cite{AkSh92} we have the following equivalences: 
\begin{enumerate}
\item A three web $W$ is an isoclinic Moufang web.
\item A three Web $W$ is an isoclinic group web.
\item A three Web $W$ is Grassmannizable and the three hypersurfaces, generating this web in a space $\mathbb{P}^{n+1}$ are hyperplanes.
\end{enumerate}
Therefore we have that the three web is also an isoclinic Moufang web.
\end{proof}

\medskip

\subsection{Algebraizable webs} 
From a projective variety a web can be constructed as follows. Let $V^k$ be an algebraic variety of dimension $k$ and of degree $d$ in a projective space $\Pp^n$ of dimension $n$. Then, $V^k$ meets a linear space $\Pp^{n-k}$ of dimension $n-k$ in $d$ points. Consider the Grassmann manifolds $G(n-k,n)$ of all $\Pp^{n-k}$'s in $\Pp^n$; their dimension is $k(n-k+1)$. 
The $\Pp^{n-k}$'s through a point of $\Pp^n$ form a submanifold of dimension $k(n-k)$ (i.e. codimension $k$), which contain $\Pp^{n-k}$. 
This shows that $V^k$ defines in $G(n-k,n)$ (or at least in a neigbourhood of it) a $d$-web of codimension $k$. The web defined by the above construction is {\it algebraic}. We refer to this construction as the Chern--Griffiths web construction~\cite{ChGr78}.   

\begin{thm}\label{T:alg}
Consider the statistical pre-Frobenius manifold. Then, there exists an algebraizable Chern--Griffiths web $W(n+1,n,r)$.
\end{thm}
\begin{proof}
 We have shown that the pre-Frobenius manifold is an algebraic variety, previously. So, using the construction above we obtain an algebraizable web. 
 \end{proof}

 We have the following statement which is that:

\begin{cor}
Consider the statistical pre-Frobenius manifold and the web $W$ obtained from the construction above.
Then, the web $W$ is hexagonal and isoclinic.
\end{cor}
\begin{proof}
By theorem. 2.6.9 \cite{Go88}: a web $W(n+1,n,r)$ is algebraizable if and only if it is isoclinic and hexagonal.
Applying the theorem~\ref{T:alg} the conclusion is direct.
\end{proof}

Statistical data is impacted by this web result. 
We can state as a corollary:
\begin{cor} 
Suppose $\Omega$ is a finite discrete sample space. 
Then, the statistical data given by $q_i(\omega_j)$, where $\omega_j\in \Omega$ lie on the (hyper)webs.
\end{cor}
\begin{proof}
This follows from the definition of an algebraic web and from the fact that the system of equations, attached to every point on the exponential statistical manifold $\cQ\bu$=0 forms a subset of a (hyper)plane. 
\end{proof}

\end{document}